\documentclass[12pt]{amsart}
\usepackage{amsmath,amssymb,amsbsy,amsfonts,amsthm,latexsym,
            amsopn,amstext,amsxtra,euscript,amscd,amsthm,graphicx,pdfpages,comment}
\usepackage[left=3.5cm,right=3.5cm,top=3cm,bottom=3cm,includeheadfoot]{geometry}
\usepackage[enableskew]{youngtab} 
\usepackage{caption}

\newtheorem{prop}{Proposition}
\newtheorem{proposition}[prop]{Proposition}

\newtheorem{thm}{Theorem}
\newtheorem{theorem}[thm]{Theorem}

\newtheorem{cor}{Corollary}
\newtheorem{corollary}[cor]{Corollary}

\theoremstyle{remark}
\newtheorem*{remark}{Remark}
\newtheorem*{Conjecture}{Conjecture}
\newtheorem{Challenge}{Challenge}
\newtheorem*{definition}{Definition}
\newtheorem{exam}{Example}
\newtheorem{example}[exam]{Example}
\newtheorem*{Acknowledgements}{Acknowledgements}

\def\\{\cr}
\def\({\left(}
\def\){\right)}
\def\<{\langle}
\def\>{\rangle}

\def\func#1{\mathop{\rm #1}}%

\def\mathscr{\mathcal}       

\begin{document}
\title{Horizontal and Vertical Log-Concavity}
\author{Bernhard Heim }
\address{Faculty of Mathematics, Computer Science, and Natural Sciences,
RWTH Aachen University, 52056 Aachen, Germany}
\email{bernhard.heim@rwth-aachen.de}
\author{Markus Neuhauser}
\address{Kutaisi International University (KIU), Youth Avenue, Turn 5/7 Kutaisi, 4600 Georgia}
\address{Faculty of Mathematics, Computer Science, and Natural Sciences,
RWTH Aachen University, 52056 Aachen, Germany}
\email{markus.neuhauser@kiu.edu.ge}
\pagenumbering{arabic}

\begin{abstract}
Horizontal and vertical generating functions and recursion relations have been
investigated by Comtet for triangular double sequences.
In this paper we investigate the horizontal and vertical log-concavity
of triangular sequences assigned to polynomials which show up in
combinatorics, number theory and physics.
This includes Laguerre polynomials, the Pochhammer polynomials, the D'Arcais and
Nekrasov--Okounkov polynomials.

\end{abstract}
\maketitle
\newpage
\section{Introduction}
Log-concavity of sequences of numbers, for example binomial coefficients and Stirling numbers,
coefficients of polynomials,
and values of discrete random variables, is an important characteristic studied
in algebra, combinatorics, computer science, number theory, probability
and statistical mechanics.

Comtet (\cite{Co74}, Chapter V), recorded horizontal and vertical recurrence relations for
triangular sequences, including Stirling numbers of the first kind and of the second kind.
This paper investigates  horizontal and vertical log-concavity for 
the Pochhammer polynomials, D'Arcais and Nekrasov--Okounkov polynomials.
In several cases the vertical log-concavity fails, but still a
vertical $C$-log-concavity
seems to be in place. We refer to a recent paper by Hong and Zhang (\cite{HZ20}, see also Section 3).

In 2003, Nekrasov and Okounkov \cite{NO03,NO06} discovered a spectacular hook length formula, which comprises 
building blocks
polynomials $Q_n(x)$ of degree $n$ involving all partitions $\lambda \vdash n$ of $n$
and the multisets of hook lengths $\mathcal{H}(\lambda)$:
\[
\label{Q}
Q_n(x) := 
\sum_{ \lambda \vdash n} 
\prod_{ h \in \mathcal{H}(\lambda)} 
\left(  \frac{h^2 + x}{h^2} \right) \in \mathbb{C}[x].
\]
Shortly after that Westbury \cite{We06} and Han \cite{Ha10} also discovered the same formula utilizing different methods.
The polynomials $Q_n(x)$ have non-negative rational coefficients and the sequence of coefficients has no internal zeros.
Applying Newton's theorem could be considered in order to prove log-concavity, if the roots are real, 
which was conjectured (see Amdeberhan \cite{Am15}).
Recently, Heim and Neuhauser \cite{HN18} disproved the conjecture.
Nevertheless, there is overwhelming numerical evidence that $Q_n(x)$ is still log-concave. We have
used PARI/GP to check this up to $n=1500$.

Recently, Hong and Zhang \cite{HZ20}, supervised by Ono, invested in the log-concavity and unimodality of $Q_n(x)$.
They proved for sufficiently large $n$ that  a huge proportion of the coefficients of $Q_n(x)$ at the 
beginning of the sequence have the log-concavity property, at the end they are
decaying.
Additionally, they came up with a conjecture which also implies the unimodality for $Q_n(x)$ for large $n$.

In this paper we define families of polynomials $\{P_n^{g,h}(x)\}_{n=1}^{\infty}$, 
where $g$ and $h$ are normalized non-vanishing arithmetic functions.
Let $P_0^{g,h}(x)=1$. Then 
\begin{equation}\label{rec}
P_n^{g,h}(x) := \frac{x}{h(n)} \sum_{k=1}^n g(k) \, P_{n-k}^{g,h}(x).
\end{equation}
Although most of the results stated in this paper, assume that  $h \in \{ \func{id},1\}$,
we keep the general notation to emphasize the connection by (\ref{rec}) and the hope to
come up with more general theorems in the future.

We are interested in particular in the arithmetic functions 
$ 1(n)=1$, $\func{id}(n)=n$, $s(n)=n^2$ and $ \sigma(n)= \sum_{d \vert n} d $.
Further, let $\tilde{g}(n):= g(n)/n$.
Note that $\tilde{1}(n)= 1/n$.
Suppose that $G(T):=\sum_{n=1}^{\infty} g(n) T^n$ is
analytic at $T=0$. We will then have 
for $h \in \{ \func{id},1\}$ generating series of
exponential and geometric type (see also Corollary~\ref{gen}):
\begin{eqnarray}
\sum_{n=0}^{\infty} P_n^{g, \func{id}}(x) \, T^n & = & \exp \left(x \sum_{n=1}^{\infty} g(n) \, \frac{T^n}{n} \right) ,\label{E}\\
\sum_{n=0}^{\infty} P_n^{g,1}(x) \, T^n & = & \frac{1}{ 1 - x \sum_{n=1}^{\infty} g(n) \, T^n}. \label{G}
\end{eqnarray}
We note that for specific arithmetic functions $g$, we obtain the Pochhammer polynomials (rising
factorials: $g=1$, $h=\func{id}$),
associated Laguerre polynomials
($g=\func{id}$, $h=\func{id}$),
Chebyshev polynomials of the second kind ($g=\func{id}$, $h=1$), and also
polynomials attached to reciprocals of
Klein's $j$-invariant ($g\left( n\right) $ is the $n-1$st coefficient of
Klein's $j$-invariant, cf.\ OEIS sequence A000521, divided by $744$, $h=1$) and reciprocals of
Eisenstein series $E_k$ of weight $k$ ($g=\sigma _{k-1}$, $h=1$)
\cite{HN19B,HLN19}.

Important examples in this paper are the D'Arcais polynomials and the related
Nekrasov--Okounkov polynomials.
Let $g(n)= \sigma(n)$ and $h(n)=\func{id}(n)=n$. Then the D'Arcais polynomials \cite{DA13, We06, HLN19}
$P_n^{\sigma, \func{id}}(x)$ are equal to the
$n$th coefficient of the $-x$th power of the Dedekind $\eta$-function \cite{On03}.
Nekrasov and Okounkov proved that
\[
Q_n(x) =  P_n^{\sigma, \func{id}}(x +1).
\]
Brenti's result \cite{Br97},
on polynomials with non-negative coefficients and where the sequence of coefficients has no internal zeros,
implies that if the D'Arcais polynomials are log-concave then the
Nekrasov--Okounkov polynomials also are log-concave.
The converse of Brenti's Theorem is in general not true and it is not generally true for unimodality. 
\newline

The purpose of this paper is three-fold. We first prove that there
is a crossover between the 
coefficients of $P_n^{\tilde{g},1}(x)$ and $P_n^{g,\func{id}}(x)$. The crossover transfers log-concavity
from $P_n^{\tilde{g},1}(x)$ to $P_n^{g,\func{id}}(x)$.
Second, suppose that $A_{n,m}^{g,h}$ is the
$m$th coefficient of $P_{n}^{g,h}(x)$. Suppose that $m$ is
fixed, then we examine
the vertical log-concavity with respect to $n$. Hong and Zhang \cite{HZ20} defined a property
which could still be true, when the vertical log-concavity fails. Let us name this new property
vertical $C$-log-concavity
(see Section 3).
Finally we
raise a question on horizontal and vertical
$C$-log-concavity on the double sequence of
coefficients of $\{P_n^{g,h}(x)\}_{n=1}^{\infty}$, where $h=1$ or $h=\func{id}$.
\section{First Results}                                              
Suppose that $g$ is
a normalized arithmetic function. We say that $g$ is of moderate growth if $G(T)$ is analytic at $T=0$.
Further, suppose that $h$ is
a normalized non-vanishing arithmetic function. We then have
for $n\geq 1$:
\[
P_n^{g,h}(x) = A_{n,n}^{g,h} x^n + A_{n,n-1}^{g,h} x^{n-1}+ \ldots + A_{n,1}^{g,h} x.
\]
Here $A_{n,n}^{g,h} = 1/\prod_{k=1}^n h(k)$ and $A_{n,1}^{g,h}= g(n)/h(n)$. Moreover, suppose that the values of $g$ and $h$
are positive integers, then
\begin{equation} \label{conv}
\left( \prod_{k=1}^n h(k) \right)  A_{n,m}^{g,h} \in \mathbb{N} \text{ for }
1 \leq m \leq n .
\end{equation}
It is difficult to obtain results like (\ref{conv}) for arbitrary pairs $(g,h)$.
Since $P_n^{\func{id}, \func{id}}\left( x\right) $ has only real roots, but some of the Nekrasov--Okounkov polynomials, e.~g.\
$Q_{10}(x) = P_{10}^{\sigma, \func{id}} (x+1)$ have also non-real roots \cite{HN18}, it is obvious, that some results
only work for specific chosen pairs. It is not clear to us if the following conversion principle can
be stated in a more general form.
\begin{theorem} [Conversion Principle] 
\label{th1}Suppose that $g$ is
a normalized arithmetic function of moderate growth. Then
\[
      A_{n,m}^{g,\func{id}} =  \frac{1}{m!} \, A_{n,m}^{\tilde{g},1}. 
\]
\end{theorem}

\begin{example}
We have $P_n^{1,1}(x) = x (x+1)^{n-1}$ for $n\geq 1$. This implies that
\begin{equation}
A_{n,m}^{1,1} = \binom{n-1}{m-1} \text{ and from
Theorem \ref{th1} that } A_{n,m}^{\func{id},\func{id}} = \frac{1}{m!} \binom{n-1}{m-1}.
\label{eq:laguerre}
\end{equation}
This further implies that
$P_n^{\func{id}, \func{id}}(x) = \frac{x}{n} L_{n-1}^{(1)}\left( -x\right) $,
where $L_n^{(\alpha)}(x)$ is
the associated Laguerre polynomial of degree $n$ with parameter $\alpha$ (\cite{Do16}, Chapter 3).
Note that (\ref{eq:laguerre})
gives conceptional proof of Theorem 2 in \cite{HN19A}; see also \cite{HLN19}.
Recall that $L_n^{(\alpha)}(x)$ for $\alpha >-1$ are orthogonal polynomials and solutions of the differential equation
\[
x \frac{ \mathrm{d}^2y}{\mathrm{d}x^2} + (\alpha + 1 -x) \frac{\mathrm{d}y}{\mathrm{d}x} + n y =0.
\]
\end{example}

Before we prove Theorem \ref{th1} we provide some applications.
\begin{corollary}\label{cor1}
Suppose that $P_n^{\tilde{g},1}(x)$ is
log-concave, then $P_n^{g,\func{id}}\,(x)$ is log-concave.
\end{corollary}
We can also fix $m$ and consider the sequence $\{ A_{n,m}^{g,h}\}_n$. If this sequence is log-concave for all $m$ we say that
$\{P_n^{g,h}(x)\}_n $ is vertically log-concave.
\begin{corollary} \label{cor2}
$\{ P_n^{\tilde{g},1}(x)\}_n $ is vertically log-concave iff $\{P_n^{g,\func{id}}\,(x)\}_n $ is vertically log-concave.
\end{corollary}

\begin{proof}[Proof of Theorem \ref{th1}]
We call $\mathcal{E}_g(T):= \sum_{n=1}^{\infty} \frac{g(n)}{n} \, T^n$ the (modified) 
Eichler
integral of $G(T)$. In the case of $g(n)= \sigma(n)$ and
$T= q := \textrm{e}^{2 \pi \textrm{i} \tau}$, where $\tau$ is in
the upper complex half-plane, $\mathcal{E}_g(T)$ is the 
Eichler
integral of $\frac{1-E_2}{24}$ of the weight $2$ quasi-modular Eisenstein series $E_2(\tau)$. We refer to \cite{BOW20} for recent
work on Eichler integrals.

We prove that for all $m \in \mathbb{N}$:
\begin{eqnarray*}
\mathcal{E}_g(T)^m & = & m!  \,\, \sum_{n=m}^{\infty} A_{n,m}^{g, \func{id}}\, T^n ,  \label{P:coefficients}\\
\mathcal{E}_g(T)^m & = & \sum_{n=m}^{\infty} A_{n,m}^{\tilde{g}, 1} \, T^n . \label{Q:coefficients}
\end{eqnarray*}
The strategy of the proof is the following. 
We consider the $x$ expansion of $\sum_{n=0}^{\infty} P_{n}^{g,h}(x)\, T^n$ for $h \in \{\func{id}, \, 1 \}$.
In the domain of absolute convergence we interchange two infinite sums and compare the coefficients.
The core of the proof is the transition from the exponential series to the geometric series.
The formula utilizing the geometric series is given by
\[
\sum_{n=0}^{\infty} P_{n}^{\tilde{g},1}(x)\, T^n =  \frac{1}{1- x \, \mathcal{E}_g(T)}.
\]
This can be directly verified by comparing the Cauchy product of the two power series
$\sum_{n=0}^{\infty} P_{n}^{\tilde{g},1}(x)\, T^n $ and $1- x \, \mathcal{E}_g(T)$ and the
defining
recursion formula of $P_{n}^{\tilde{g},1}(x)$.
\newline
The formula utilizing the exponential series is given by
\[
\sum_{n=0}^{\infty} P_{n}^{g, \func{id}}(x)\, T^n =  \exp \left( x \, \mathcal{E}_g(T) \right).
\]
Suppose that the generating series
$\sum_{n=0}^{\infty} P_{n}^{g, \func{id}}(x)\, T^n $ is denoted by $F_g(x,T)$.
First, we observe that the recursion formula 
(\ref{rec}) is encoded in the functional equation
\begin{equation}
\frac{\partial}{\partial T} F_g(x,T) = F_g(x,T) \, \frac{\mathrm{d}}{\mathrm{d}T}\left( x\mathcal{E}_{g}\left( T\right) \right) .
\label{functional}
\end{equation}
Further, let $f(n):= \sum_{d \vert n} \mu(d) \, g(n/d)$, where $\mu$ is the Moebius function.
Then it can be shown by a standard procedure (logarithmic differentiation), that the Euler product
\[
\prod_{n=1}^{\infty} \left( 1 - T^n \right)^{- \frac{x f(n)}{n}}
\]
satisfies the functional equation (\ref{functional}). Finally, since 
\[
\ln \prod_{n=1}^{\infty} \left( 1 - T^n \right)^{- \frac{x f(n)}{n}} = 
\sum_{n,m=1}^{\infty} \frac{x \, f(n)}{n} \frac{T^{nm}}{m} = x \mathcal{E}_g(T),
\]
the proof has been completed.
\end{proof}
From this proof we obtain the exponential and geometric realization of $P_n^{g,h}(x)$.
\begin{corollary}\label{gen}
Let $h=\func{id}$ or $h=1$, then we have identities
(\ref{E}) and (\ref{G}) for the assigned generating series:
\begin{eqnarray*}
\sum_{n=0}^{\infty} P_n^{g, \func{id}}(x) \, T^n & = & \exp \left(x \sum_{n=1}^{\infty} g(n) \, \frac{T^n}{n} \right) ,\label{EE}\\
\sum_{n=0}^{\infty} P_n^{g,1}(x) \, T^n & = & \frac{1}{ 1 - x \sum_{n=1}^{\infty} g(n) \, T^n}. \label{GG}
\end{eqnarray*}
\end{corollary}


\section{Log-Concavity and Double Sequences}

\begin{definition}
Suppose that $\left\{a_n\right\}_{n=0}^{\infty}$ is
a sequence of non-negative real numbers.
A finite sequence is extended with zeros.

\begin{enumerate}
\item  The sequence is
called log-concave if $a_n^2  \geq a_{n-1} \, a_{n+1}$ for $n \geq 1$.

\item  We denote a double sequence $\mathcal{A}=\{a_{n,m}\}$
{\it horizontally log-concave} iff for every $n_0 \in \mathbb{N}$
the sequence $\{a_{n_0,m}\}$ is log-concave and
{\it vertically log-concave} iff for every $m_0 \in \mathbb{N}$
the sequence $\{a_{n,m_0}\}$ is log-concave. 

\item  If for $C>1$ fixed and
for all $m_0$ the sequence 
$\{a_{n,m_0}\}_{1 \leq n \leq C^{m_0}}$ is log-concave,
then we denote the double sequence $\mathcal{A}$ as
{\it vertically $C$-log-concave}.
\end{enumerate}
\end{definition}

We are mainly interested in double sequences 
$\mathcal{A}=\{a_{n,m}\}$ of triangular shape: 
$a_{n,m}=0$ for $m>n$ or $m=0$,
and $a_{n,m}$ are otherwise positive.

Suppose that $g$ and $h$ are
normalized arithmetic functions with positive real values. 
Suppose that $g$ is of moderate growth.
Then we assign to the family of polynomials $P_n^{g,h}(x)$ the double sequence 
$\mathcal{A}^{g,h}=\{a_{n,m}^{g,h}\}$ of triangular shape by putting 
\[
a_{n,m}^{g,h} = A_{n,m}^{g,h} \quad \text{ for } 1 \leq m \leq n,
\]
and otherwise zero. If $\mathcal{A}^{g,h}$ is horizontally or vertically
log-concave or
vertically $C$-log-concave we give
$P_n^{g,h}(x)$ the same label.
We are interested in the D'Arcais polynomials $P_n^{\sigma, \func{id}}(x)$ which
are conjectured to be horizontally log-concave. In the next section we provide counter-examples for
vertical log-concavity and put this observation in the context of a conjecture by Hong and Zhang \cite{HZ20}, addressing
the vertical $C$-log-concavity.

Let us first give some examples, which may serve as a source
of ideas to prove the Hong and Zhang conjecture.
Example 1 leads to:
\begin{proposition}
Let $g(n)=h(n)$ equal to $1(n)$ or
$\func{id}\left( n\right) $. Then the assigned double sequences  $\mathcal{A}^{1,1}$ and $\mathcal{A}^{\func{id}, \func{id}}$ 
are horizontally and vertically log-concave.
\end{proposition}
\begin{proof}
We recall that 
\[
a_{n,m}^{1,1} = \binom{n-1}{m-1}.
\]
Then $\mathcal{A}^{1,1}$ consists of binomial coefficients, and therefore are
(horizontally) log-concave.
Binomial coefficients are also vertically log-concave:
\[
{\binom{n}{k}}^2 \geq \binom{n+1}{k} \, \binom{n-1}{k}.
\]
Horizontal and vertical log-concavity of $\mathcal{A}^{1,1}$ implies by Corollary \ref{cor1} and Corollary \ref{cor2} the proof of the
proposition. This could also be obtained by a direct calculation.
\end{proof}

\begin{example}
Let $s\left( n\right) =n^2$. We recall from \cite{HN19A} that 
$$ A_{n,m}^{s, \func{id}} = \frac{1}{m!} \binom{n+m-1}{2m-1}.$$
\end{example}

We obtain:
\begin{proposition}
The double sequences  $\mathcal{A}^{\func{id},1}$ and $\mathcal{A}^{s, \func{id}}$ 
are horizontally and vertically log-concave.
\end{proposition}
\begin{proof}
It is sufficient to show that the double sequence $\{ \binom{n+m-1}{2m-1} \}_{n,m}$ is horizontally and vertically log-concave.
This is shown in a straightforward manner.
\end{proof}

\begin{example}
The polynomials $n! \, P_n^{1,\func{id}}(x)$ are obtained by the raising factorials. We have
\[
P_n^{1, \func{id}}(x) = \frac{ x (x+1) \ldots (x+n-1)}{n!}.
\label{eq:stirlingpolynomials}
\]
Suppose that $S\left( n,m \right) =
 \left[\begin{array}{c} n \\ m
\end{array}
\right] $ 
is the unsigned
Stirling number of the first kind and ${\widetilde{1}}(n) = \frac{1}{n}$. Then
\[
A_{n,m}^{1, \func{id}} = \frac{1}{n!} 
\left[\begin{array}{c} n \\ m
\end{array} \right] 
\text{ and } 
A_{n,m}^{{\widetilde{1}}, 1} = \frac{m!}{n!} \left[\begin{array}{c} n \\ m
\end{array} \right]. 
\]
Recall that the unsigned Stirling numbers of the first kind
satisfy a three term recursion formula,
similar to the recursion formula of binomial coefficients:
\[
S(n,m) = (n-1) S(n-1,m) + S(n-1,m-1).
\]
\end{example}

This example is quite interesting. As a first result we have:
\begin{proposition}
The double sequences $\mathcal{A}^{1, \func{id}}$ and 
$\mathcal{A}^{\tilde{1}, 1}$ 
are horizontally log-concave but not vertically log-concave.
\end{proposition}
\begin{proof}
The double sequence $\mathcal{A}^{1, \func{id}}$ is horizontally
log-concave according to the theorem by Newton.
To prove that the double sequence $\mathcal{A}^{\tilde{{1}}, 1}$ 
is horizontally log-concave we cannot apply
Newton's theorem, since we do not know if 
$P_n^{\tilde{{1}},1}(x)$ has only real roots.
$\mathcal{A}^{\tilde{{1}}, 1}$ is horizontally
log-concave iff
\begin{equation}
m  S(n,m)^2 \geq (m+1) \,S(n,m+1) \, S(n,m-1).
\label{eq:horizontal1}
\end{equation}
Since $S\left( n,m\right) =0$ only for $m\leq 0$ or
$m\geq n+1$, the equation (\ref{eq:horizontal1}) is
true for $m\leq 1$ and $m\geq n$.

For the rest of this part of the proof we now assume
$2\leq m\leq n-1$, which implies that all in  (\ref{eq:horizontal1})
involved Stirling numbers are non-zero. From this we obtain
\begin{equation}
mS\left( n,m\right) /S\left( n,m-1\right) \geq \left( m+1\right) S\left( n,m+1\right) /S\left( n,m\right).
\label{eq:horizontal2}
\end{equation}
Sibuya (\cite{Si88}, Corollary~3.1) has shown that
(\ref{eq:horizontal2})
holds strictly which proves our
claim.

To disprove the vertical log-concavity, we record the coefficients of the first six polynomials 
$P_n^{1, \func{id}}(x)$.
\[
\begin{array}{|c|c|c|c|c|c|c|}
\hline
 n! \, {A}_{n,m}^{1, \func{id}} & m=6 & m=5&m=4 & m=3 & m=2 & m=1 \\ \hline
n=6 & 1&  15&85  & 225 & 274 & 120\\ \hline

n=5 & 0&1&10  & 35 & 50 & 24\\ \hline

n=4 & 0&0& 1 & 6 & 11 & 6\\ \hline

n=3 & 0&0& 0 & 1 & 3 & 2\\ \hline
n=2 &0& 0& 0 & 0 & 1  & 1 \\ \hline
n=1 &0&  0&   0 & 0 & 0 & 1 \\ \hline
\end{array}
\]
Let $m=1$, then $n_0=2$ is the smallest $n$, such that
\begin{equation}\label{fail m=1}
\left( A_{n,1}^{\tilde{{1}}, 1}\right)^2 
\geq 
A_{n+1,1}^{\tilde{{1}}, 1} \,\,
A_{n-1,1}^{\tilde{{1}}, 1}.
\end{equation}
fails. Moreover, (\ref{fail m=1}) fails for all $n \geq 2$.

Let $m=2$, then $n=n_0=5$ is the smallest $n$, such that 
\begin{equation}\label{fail m=2}
\left( A_{n,2}^{{\tilde{1}}, 1}\right)^2 
\geq
A_{n+1,2}^{\tilde{{1}}, 1} \,\,
A_{n-1,2}^{\tilde{{1}}, 1}.
\end{equation}
fails.
\end{proof}

\begin{remark}
From (\cite{HNT20}, Example~2.3)
we obtain
$P_{n+1}^{1,\func{id}}\left( x\right) =\frac{n+x}{n+1}P_{n}^{1,\func{id}}\left( x\right) $.
There is no obvious finite recursion formula for
$P_{n}^{\tilde{{1}},1}\left( x\right) $.
\end{remark}

In the following we show that also in the case of $m=2$ the inequality (\ref{fail m=2}) fails
for all $n \geq n_0=5$.
Let $H(n)=\sum_{k=1}^n \frac{1}{k}$ denote the $n$th Harmonic number.
The sequence of Harmonic numbers is log-concave. It is easy to see that
\[
S(n,1)= (n-1)! \text{ and } S(n,2) = (n-1)! \,\, H(n-1).
\]
Thus, $A_{n,2}^{\tilde{{1}}, 1} = \frac{2}{n} \, H(n-1)$. We study the function $$\Delta(n):= (n^2-1) H(n-1)^2 - n^2 H(n) H(n-2).$$
Since $H(n)= H(n-1) + \frac{1}{n}$, we obtain:
\[
\label{quad}
\Delta(n) = \frac{n}{n-1} ( H(n-1)+ 1) - H(n-1)^2.
\]
Let $ n \geq 2$.
We have for $2\leq n\leq 4$ the bounds $1\leq H\left( n-1\right) <2$. Thus,
\begin{eqnarray*}
\Delta \left( n\right) &\geq &\frac{4}{3}\left( H\left( n-1\right) +1\right) -\left( H\left( n-1\right) \right) ^{2}\\
&=&\frac{1}{3}\left( 2+3H\left( n-1\right) \right) \left( 2-H\left( n-1\right) \right) >0.
\end{eqnarray*}
Otherwise, $H(n-1) >2$, which leads to
\begin{eqnarray*}
\Delta \left( n\right) &\leq &\frac{4}{3}\left( H\left( n-1\right) +1\right) -\left( H\left( n-1\right) \right) ^{2}\\
&=&\frac{1}{3}\left( 2+3H\left( n-1\right) \right) \left( 2-H\left( n-1\right) \right) <0.
\end{eqnarray*}

\section{On Hong and Zhang's Conjecture}

Recently, Hong and Zhang \cite{HZ20} presented an interesting conjecture which implies
that the Nekrasov--Okounkov polynomials are unimodal for large $n$.

\subsection{Nekrasov--Okounkov polynomials}
In 2003 Nekrasov and Okounkov, in an arXiv preprint (\cite{NO03}, formula (6.12)), announced a remarkable new type
of hook length formula, based on their work on random partitions and the Seiberg--Witten theory (final publication \cite{NO06}).
Shortly after their discovery,  Westbury \cite{We06} and Han \cite{Ha10} also
spotted the hook length formula in connection with the Macdonald identities.
Westbury utilized properties of twisted universal characters and Han properties of $t$-cores.

The hook length formula relates the sum over products of partition hook lengths \cite{Ma95,Fu97} to 
the coefficients of complex powers of Euler products \cite{Ne55,Se85,HNW18}, which is essentially a power of the Dedekind eta function.

Suppose that $\lambda$ is
a partition of $n$ denoted by $\lambda \vdash n$ with weight $|\lambda|=n$. We denote by
$\mathcal{H}(\lambda)$ the multiset of hook lengths associated to $\lambda$ and by $\mathcal{P}$ the set of all partitions.
The Nekrasov--Okounkov hook length formula is given by
\begin{equation} \label{ON}
\sum_{ \lambda \in \mathcal{P}} q^{|\lambda|} \prod_{ h \in \mathcal{H}(\lambda)} 
\left(  1 - \frac{z}{h^2} \right) =   \prod_{m=1}^{\infty} \left( 1 - q^m \right)^{z-1}.
\end{equation}
The identity (\ref{ON}) is valid for all $z \in \mathbb{C}$.
The Dedekind eta function $\eta(\tau)$ is given by $q^{\frac{1}{24}} \prod_{m=1}^{\infty} \left( 1 - q^m \right)$ (see \cite{On03}).
In \cite{HN18} we revised and refined three conjectures posted by Amdeberhan \cite{Am15}. 
The formula (\ref{ON}) is built up of a family of polynomials $Q_n(x)$.
The $n$th Nekrasov--Okounkov polynomial is given by
\begin{equation}\label{Q}
Q_n(x) := 
\sum_{ \lambda \vdash n} 
\prod_{ h \in \mathcal{H}(\lambda)} 
\left(  \frac{h^2 + x}{h^2} \right) \in \mathbb{C}[x].
\end{equation}
Note that all the coefficients of $Q_n(x)$ are non-negative.
Suppose that $Q_n(x)$ has only real roots (previous Conjecture 2), then we know already
from Newton that this implies that
$Q_n(x)$ is log-concave, and hence unimodal (previous Conjecture 3).
We disproved Conjecture 2, and used PARI/GP to check
that $Q_n(x)$ is log-concave
for $n \leq 1500$.
Note that the roots of $Q_n(x)$ are directly related to the Lehmer conjecture \cite{Le47, HNW18}.

\subsection{D'Arcais polynomials}
In 1913 D'Arcais \cite{DA13} studied a sequence of polynomials $P_n(x)$ (which are denoted D'Arcais polynomials \cite{We06}):
\[
\sum_{n=0}^{\infty} P_{n}\left( x\right) \, q^{n}
=  \prod_{n=1
}^{\infty} \left( 1 - q^n \right)^{-x}. \label{Arcais}
\]
The coefficients are called D'Arcais numbers \cite{Co74}.
Newman and Serre \cite{Ne55, Se85} studied the polynomials in the context of modular forms.
Serre proved his famous theorem on lacunary modular forms, utilizing the factorization of $P_n(x)$ for $ 1 \leq n \leq 10$ over $\mathbb{Q}$.

\subsection{Hong and Zhang's conjecture}
Hong and Zhang \cite{HZ20} investigated
a conjecture, published by Heim and Neuhauser \cite{HN18} on the 
unimodality of the Nekrasov--Okounkov polynomials and the related log-concavity (see also \cite{Am15}).

We consider the generating series $$f(q):= \sum_{n=1}^{\infty} \tilde{\sigma}(n) \, q^n = 
\sum_{n=1}^{\infty} \sigma(n) \, \frac{q^n}{n}.$$ 
Let $m \in \mathbb{N}$. 
We denote by $b_{m,n}$ the coefficients
of the $q$-expansion of the $m$th power of $f(q)$ and we put otherwise $b_{m,n}=0$.
We follow \cite{HZ20} with $m=k$ and $b_{m,n}= c_{n,k}$.
\newline
\newline

\begin{Conjecture}[Hong, Zhang]
There exists a constant $C>1$, such that for all $m \geq 2$ and $1 \leq n \leq C^m$:
\[
b_{m,n}^2 \geq b_{m,n-1} \, b_{m,n+1}.
\]
\end{Conjecture}

At the end of their paper they offer evidence and also remark that it is very likely that $C=2$ may fulfill their
conjecture. Finally, they prove that the validity of their conjecture implies
that
the Nekrasov--Okounkov polynomials are unimodal for large degrees.
\subsection{Vertical version of the Hong--Zhang conjecture}
The following identity puts the numbers
$b_{m,n}$ in the context of double sequences 
derived by polynomials defined by the recursion (\ref{rec}). 
We consider the $q$-expansion of reciprocal power series \cite{HN19B}, induced by normalized arithmetic functions $g$,
which provides the link to the polynomials $P_n^{g,1}(x)$:
\[
\frac{1}{1 - x \sum_{n=1}^{\infty} g(n) \, q^n} = \sum_{n=0}^{\infty} P_n^{g, 1}(x) \, q^n.
\]
Thus, the numbers $b_{m,n}$ are the $m$th
coefficients of the polynomial $P_n^{\tilde{\sigma}, 1} (x)$.
The duality property shows that $b_{m,n}/m!$ is also equal to the
$m$th coefficients $A_{n,m}$ of the D'Arcais polynomials.
Hence, we immediately obtain:
\begin{theorem}
Let $C>1$. The Hong--Zhang conjecture with $C>1$ is true 
iff the double sequence
$\mathcal{A}^{\sigma, \func{id}}$ attached to the D'Arcais polynomials is
vertically $C$-log-concave 
iff the double sequence
$\mathcal{A}^{\widetilde{\sigma}, 1}$ is vertically $C$-log-concave. 
\end{theorem}
\section{Open Challenges and Further Study}

Suppose that $g$ is
a normalized positive real-valued arithmetic function of moderate growth.
We consider double sequences $\mathcal{A}^{g,h}$ assigned to families of polynomials $P_n^{g,\func{id}}(x)$ and
$P_n^{g,1}(x)$ for $h=\func{id}$ or $h=1$.
The horizontal and vertical log-concavity of sequences are important characteristics.
We utilize a conversion principle, Theorem \ref{th1}, which provides an explicit translation
from the double sequences $\mathcal{A}^{g, \func{id} }$ to the double sequence $\mathcal{A}^{\tilde{g},1}$.

\subsection{Unsigned Stirling numbers of the first kind}
Let $g=1$ and $h={\func{id}}$. Then the double sequence $\mathcal{A}^{1, \func{id}}$ assigned to
\[
P_n^{1,\func{id}}(x) = \frac{1}{n!}  \prod_{k=0}^n (x +k) =  \frac{1}{n!} \sum_{m=1}^{n}   S(n,m) \, x^m
\]
is horizontal log-concave. We have already shown that the vertical log-concavity fails. Nevertheless,
numerical investigation shows similarities to the Hong--Zhang conjecture. 
We have recorded this in the following Table 1.

\begin{center}
\begin{minipage}[t]{1.0 \textwidth}
\[
\begin{array}{|l||c|c|c|c|c|c|c|}
\hline
m & 1 & 2 & 3 & 4 & 5 & 6 & 7 \\ \hline
n_{0} & 2 & 5 & 17 & 54 & 162 & 469 & 1330 \\ \hline
\end{array}
\]
\captionof{table}{Smallest integers $n=n_{0}$ where 
$\left( \frac{S\left( n,m\right) }{n!}\right) ^{2}<\frac{S\left( n-1,m\right) S\left( n+1,m\right) }{\left( n-1\right) !\left( n+1\right) !}$.}
\label{stirling}
\end{minipage}
\end{center}

\begin{proposition}
For large $n$ the sequence $\frac{S\left( n,m\right) }{n!}$ is not
log-concave.
\end{proposition}

In the proof we will use
Landau's big $O$ notation for functions on an unbounded subset of
the positive real
numbers
and $f\left( x\right) =O\left( g\left( x\right) \right) $ if there is a $C>0$
and an $x_{0}$
such that $\left| f\left( x\right) \right| \leq Cg\left( x\right) $ for
all $x\geq x_{0}$. We also will need the
polynomials $v_{m}\left( y\right) $
of degree $m$ defined by the
power series expansion of
$\frac{\textrm{e}^{xy}}{x\Gamma \left( x\right) }=\sum _{m=0}^{\infty }v_{m}\left( y\right) x^{m}$.

\begin{proof}
From \cite{Wi93}
we can deduce that
$\frac{S\left( n,m\right) }{n!}=\frac{v_{m-1}\left( \ln \left( n-1\right) \right) }{n}+O\left( n^{\varepsilon /2-2}\right) $
for any $0<\varepsilon <1$.
Then
\begin{eqnarray*}
&&\left( \frac{S\left( n,m\right) }{n!}\right) ^{2}-\frac{S\left( n-1,m\right) S\left( n+1,m\right) }{\left( n-1\right) !
\left( n+1\right) !}\\
&=&\left( \frac{v_{m-1}\left( \ln \left( n-1\right) \right) }{n}\right) ^{2}-\frac{v_{m-1}\left( \ln \left( n-2\right) \right) }{n-1}\frac{v_{m-1}\left( \ln \left( n\right) \right) }{n+1}+O\left( n^{\varepsilon -3}\right) .
\label{eq:logconcavequotient}
\end{eqnarray*}
Now
$\ln \left( n-1\pm 1\right) =\ln \left( n-1\right) 
+\ln \left( 1\pm \frac{1}{n-1}\right) =\ln \left( n-1\right) +O\left( n^{-1}\right) $.
Since $v_{m-1}\left( y\right) $ is a polynomial we have
$v_{m-1}\left( \ln \left( n-1\pm 1\right) \right) =v_{m-1}\left( \ln \left( n-1\right) \right) +O\left( n^{-1}\right) $.
Together
this yields
\begin{eqnarray*}
&&\left( \frac{S\left( n,m\right) }{n!}\right) ^{2}-\frac{S\left( n-1,m\right) S\left( n+1,m\right) }{\left( n-1\right) !\left( n+1\right) !}\\
&=&\left( \frac{1}{n^{2}}-\frac{1}{n^{2}-1}\right) \left( v_{m-1}\left( \ln \left( n-1\right) \right) \right) ^{2}+O\left( n^{\varepsilon -3}\right) .
\label{eq:stirlingasymptotisch}
\end{eqnarray*}
Therefore its sign
is determined by the sign of
\[
\left( \frac{1}{n^{2}}-\frac{1}{n^{2}-1}\right) \left( v_{m-1}\left( \ln \left( n-1\right) \right) \right) ^{2}<0.
\]
\end{proof}

\begin{Challenge}
Is the double sequence $\mathcal{A}^{1, \func{id}}$ vertically $C$-log-concave?
\end{Challenge}

We expect that the answer to this question will also give some insight in the Hong--Zhang conjecture.
\subsection{D'Arcais polynomial version of Hong--Zhang's conjecture}
Hong and Zhang \cite{HZ20} considered the coefficients $b_{m,n}$ of the $m$th power of the generating series 
$$f(q)= 
\sum_{n=1}^{\infty} \sigma(n) \, \frac{q^n}{n}$$
and conjectured that there exists a $C>1$, such that for $m \geq 2$ and $1 \leq n \leq C^m$:
\[
b_{m,n}^2 \geq b_{m,n-1} \, b_{m,n+1}.
\]
We have proven in this paper that the conjecture is equivalent to the vertical
$C$-log-concavity of the D'Arcais polynomials $P_n^{\sigma, \func{id}}(x)$ given by
\[
\sum_{n=0}^{\infty} P_n^{\sigma, \func{id}}\left( x\right) \, q^n = \prod_{n=1}^{\infty} \left( 1 - q^n \right)^{-x}.
\]

\begin{Challenge}
Assume $\mathcal{A}^{\sigma, 
\func{id}}$ horizontally log-concave. Does this imply the Hong--Zhang conjecture?
\end{Challenge}

\subsection{D'Arcais polynomials}
The Nekrasov--Okounkov $Q_n(x)$ polynomials are shifted D'Arcais polynomials.
We have checked numerically that for $n \leq 1500$ the D'Arcais polynomials are
(horizontally) log-concave.
We would reinforce the conjecture stated in \cite{Am15} and \cite{HN18} on the unimodality of
the Nekrasov--Okounkov polynomials.

\begin{Challenge}[Conjecture]
The double sequence $\mathcal{A}^{\sigma, \func{id}}$ assigned to the
D'Arcais polynomials is horizontally log-concave.
\end{Challenge}

By Brenti's result this implies the log-concavity of the Nekrasov--Okounkov polynomials and therefore also the
unimodality.
\subsection{Examples for horizontal and vertical $C$-log-concavity}
To get a better understanding of the D'Arcais polynomials it would be beneficial to know which property of
the input function $g=\sigma$
enforces the horizontal and vertical properties of the double sequences
$\mathcal{A}^{\sigma, \func{id}}$ and $\mathcal{A}^{\tilde{\sigma}, \func{1}}$.

\begin{Challenge}
Characterize the set of normalized positive-valued arithmetic functions, 
which provide horizontally and vertically
($C$-)log-concave properties of the assigned double sequences.
\end{Challenge}




\begin{Acknowledgements}
The authors thank the two anonymous referees for many helpful comments.
\end{Acknowledgements}


\end{document}